\documentclass{amsart}

\usepackage{latexsym,amsfonts} 
\usepackage{amsmath,amssymb,graphics,setspace}
\usepackage{amsthm}
\usepackage{MnSymbol}
\usepackage{mathrsfs} 
\usepackage{fontenc}%

\usepackage{paralist}
\usepackage{color}

\usepackage[verbose]{wrapfig}

\usepackage[ruled,vlined,linesnumbered]{algorithm2e}

\usepackage{marvosym}

\usepackage{picins,graphicx}

\usepackage{pstricks,pst-text,pst-grad,pst-node,pst-3dplot,pstricks-add,pst-poly} 

\usepackage{subfigure}
\renewcommand{\thesubfigure}{\thefigure.\arabic{subfigure}}
\makeatletter
\renewcommand{\p@subfigure}{}
\renewcommand{\@thesubfigure}{\thesubfigure:\hskip\subfiglabelskip}
\makeatother

\usepackage{hyperref}
\hypersetup{linktocpage=true,colorlinks=true,linkcolor=blue,citecolor=blue,pdfstartview={XYZ 1000 1000 1}}

\usepackage{floatflt,graphicx}

\newcommand{\Cech}{\u{C}ech}

\newcommand{\norm}[1]{\left\|#1\right\|}  

\newcommand{\Int}{\mbox{int}}
\newcommand{\bdy}{\mbox{bdy}}

\newcommand{\Nrv}{\mbox{Nrv}}
\newcommand{\near}{\delta} 
\newcommand{\dnear}{\delta_{\Phi}} 
\newcommand{\dcap}{\mathop{\cap}\limits_{\Phi}} 

\newcommand{\cx}{\mbox{cx}}

\newcommand{\dfar}{{\not\delta}_{\Phi}} 
\newcommand{\sn}{\mathop{\delta}\limits^{\doublewedge}} 

\newcommand{\snd}{\mathop{\delta_{_{\Phi}}}\limits^{\doublewedge}} 
\newcommand{\sdfar}{\stackrel{\not{\text{\normalsize$\delta$}}}{\text{\tiny$\doublevee$}}_{_{\Phi}}} %
\newcommand{\sfar}{\stackrel{\not{\text{\normalsize$\delta$}}}{\text{\tiny$\doublevee$}}} 
\newcommand{\notfar}{\mathop{\not{\delta}}\limits^{\doublewedge}} 

\newtheorem{example}{Example}
\newtheorem{remark}{Remark}
\newtheorem{definition}{Definition}
\newtheorem{lemma}{Lemma}
\newtheorem{theorem}{Theorem}
\newtheorem{proposition}{Proposition}
\newtheorem{corollary}{Corollary}
\newtheorem{notation}{Notation}

\usepackage{xcolor}
\definecolor{light}{gray}{0.80}

\begin{document}

\setlength{\intextsep}{0pt}
\begin{wrapfigure}[8]{R}{0.36\textwidth}
\begin{minipage}{4.4 cm}
\centering
\begin{pspicture}
(-0.2,-0.5)(5.0,4.2)
\psframe[linecolor=black](-0.5,-0.0)(4.3,4.0)
\pscircle[fillstyle=solid,linecolor=gray!55,opacity=0.5](2.5,1.8){1.2}
\pscircle[fillstyle=solid,fillcolor=gray!75,opacity=0.5](2.5,2.5){1.2}
\pscircle[fillstyle=solid,fillcolor=gray!35,opacity=0.5](1.5,1.5){1.2}
\rput(-0.25,3.7){$\boldsymbol{K}$}
\rput(0.2,2.4){$\boldsymbol{B_r(p)}$}
\rput(1.2,3.5){$\boldsymbol{B_r(q)}$}
\rput(3.5,0.5){$\boldsymbol{B_r(s)}$}
\end{pspicture}
\caption[]{Balls}
\label{fig:CechComplex}
\end{minipage}
\end{wrapfigure}

\title[Proximal Planar \u{C}ech Nerves]{Proximal Planar \u{C}ech Nerves.\\  An Approach to
Approximating the Shapes of Irregular, Finite, Bounded Planar Regions}

\author[J.F. Peters]{J.F. Peters}
\email{James.Peters3@umanitoba.ca}
\address{\llap{$^{\alpha}$\,}
University of Manitoba, WPG, MB, R3T 5V6, Canada and
Department of Mathematics, Faculty of Arts and Sciences, Ad\.{i}yaman University, 02040 Ad\.{i}yaman, Turkey}
\thanks{The research has been supported by the Natural Sciences \&
Engineering Research Council of Canada (NSERC) discovery grant 185986 
and Instituto Nazionale di Alta Matematica (INdAM) Francesco Severi, Gruppo Nazionale per le Strutture Algebriche, Geometriche e Loro Applicazioni grant 9 920160 000362, n.prot U 2016/000036.}


\date{}

\dedicatory{Dedicated to E. \Cech\ and Som Naimpally}

\begin{abstract}
This article introduces proximal \u{C}ech\ nerves and \u{C}ech\ complexes, restricted to finite, bounded regions $K$ of the Euclidean plane.  A \u{C}ech\ nerve is a collection of intersecting balls.  A \u{C}ech\ complex is a collection of nerves that cover $K$. \u{C}ech\ nerves are proximal, provided the nerves are close to each other, either spatially or descriptively.  A \u{C}ech\ nerve has an advantage over the usual Alexandroff nerve, since we need only identify the center and fixed radius of each ball in a \u{C}ech\ nerve instead of identifying the three vertices of intersecting filled triangles (2-simplexes) in an Alexandroff nerve.  As a result, \u{C}ech\ nerves more easily cover $K$ and facilitate approximation of the shapes of irregular finite, bounded planar regions.  A main result of this article is extension of the Edelsbrunner-Harer Nerve Theorem for descriptive and non-descriptive \u{C}ech\ nerves and \u{C}ech\ complexes, covering $K$.
\end{abstract}

\maketitle

\section{Introduction}
\Cech\ complexes were introduced by E. \Cech~\cite[\S A.5]{Cech1966} during a 1936-1939 Brno seminar.  In keeping with more recent work, a \Cech\ complex of a finite set of points $K\in \mathbb{R}^2$ is a collection of intersecting convex sets that are closed geometric balls, each with radius $r$~\cite[\S III.2]{Edelsbrunner1999},~\cite[\S 2.2.2]{Munch2013PhDpersistentHomology}.
Let $X=\mathbb{R}^2$ be the Euclidean plane, $K\in 2^X$, a finite, bounded plane region, and let $\Nrv A$ denote a nerve, defined by
$\Nrv A = \left\{E\subseteq K: \bigcap E\neq \emptyset\right\}$.
A closed ball $B_r(x)$\\ with center $x\in X$ 
and with radius $r > 0$ is defined\\ by
\begin{align*}
B_r(x) = \left\{y\in X: \norm{x - y}\leq r\right\}.\qquad\qquad\qquad\qquad\qquad\qquad\qquad
\end{align*}
A \Cech\ nerve on $K\in 2^X$ (denoted by \Cech $(r)$) is a collection of intersecting balls, {\em i.e.},
\[
\mbox{\Cech}(r) = \Nrv\left\{B_r(x): x\in K\right\} = \left\{B_r(x)\subseteq K: \bigcap B_r(x)\neq \emptyset\right\}.
\]

\begin{notation}
In addition to the usual \Cech $(r)$ nerve, several other forms of \Cech\ nerves are introduced here.  Let
\Cech$_r(A)$ be a \Cech\ nerve on $A\subset K$, defined by
\[
\mbox{\u{C}ech}_r(A) = \Nrv\left\{B_r(x): x\in A\right\}.
\]
Let $2^{\mbox{\u{C}ech}_r(A)}$ be a collection of \Cech\ nerves, $A\in 2^K$ in the collection of subsets in $K$ and let $\cx A$ be a \mbox{\u{C}ech} complex that covers $K$ for $A\in 2^K$, {\em i.e.},
\begin{align*}
\cx A &= 2^{\mbox{\u{C}ech}_r(A)},  A\in 2^K:\\
K &\subseteq \cx A\ \mbox{(covering of $K$ by a \Cech\ complex $\cx A$)}.
\end{align*}
In the sequel, a descriptive \Cech\ nerve (denoted by \Cech$_{r,\Phi}(K)$) on $K$ is also introduced.
\qquad \textcolor{blue}{\Squaresteel}
\end{notation}

\begin{example}
Let $K\in 2^X$ be a finite bounded subset of the Euclidean plane.
Let \Cech$(r)$ be a \Cech\ nerve, which is a collection of three intersecting closed geometric balls of points in a plane region $K$ ($B_r(p),B_r(q),B_r(s)$, each with radius $r$ and centers $p,q,s$ (not shown)), represented by shaded disks in Fig.~\ref{fig:CechComplex}, where
\begin{align*}
{p,q,s} &\subset K,\\
\mbox{\Cech}(r) &= \Nrv\left\{B_r(x): x\in {p,q,s}\right\}\\
                  &= \left\{B_r(x)\subseteq K: \mathop{\bigcap}\limits_{x\in \left\{p,q,s\right\}} B_r(x)\neq \emptyset\right\}.\ \mbox{\qquad \textcolor{blue}{\Squaresteel}}
\end{align*}
\end{example}

\noindent A \u{C}ech\ nerve on $A\subseteq K$ is denoted by $\mbox{\u{C}ech}_r(A)$. A \mbox{\u{C}ech} complex is a collection of nerves the cover $A\subseteq K$ (denoted by $\cx A$).

\begin{remark}
An important assumption made here is that the covering condition\footnote{Many thanks to the reviewer for pointing this out.} is satisfied, {\em i.e.}, each finite, bounded subset $K$ in the Euclidean plane is covered by a collection of nerves, {\em i.e.},
\begin{align*}
K &\subseteq \mathop{\bigcup}\limits_{A\subseteq K} \mbox{\u{C}ech}_r(A),\\
K &\subseteq \cx A: A\in 2^K.
\end{align*}
For more about this, see~\cite{Goaoc2015SoCGgoodCoveringCondition}. \qquad \textcolor{blue}{\Squaresteel}
\end{remark}

The study of nerves in complexes was introduced by P. Alexandroff~\cite{Alexandroff1928dimensionsbegriff},~\cite[\S 33, p. 39]{Alexandroff1932elementaryConcepts}, elaborated by J. Leray~\cite{Leray1945nerveTheorem},\cite{Leray1946homology}, K. Borsuk~\cite{Borsuk1948FMsimplexes} and a number of others such as R. Ghrist~\cite[\S 2.9, p. 31]{Ghrist2014elemAppliedTopology}, A. Hatcher~\cite[\S 3.3, p. 257]{Hatcher2002AlgebraicTopology}, A. Bj\"{o}rner~\cite[\S 4]{Bjorner2003JCTnerves}, M. Adamaszek et al.~\cite{Adamaszek2014arXivNerveComplexes}, E.C. de Verdi\`{e}re et al.~\cite{Colin2012multinerves}, H. Edelsbrunner and J.L. Harer~\cite{Edelsbrunner1999}, and more recently by M. Adamaszek, H. Adams, F. Frick, C. Peterson and C. Previte-Johnson~\cite{Adamaszek2016DCGnerveComplexesOfCircularArcs}.  In this paper, an extension of the Edelsbrunner-Harer Nerve Theorem is given.

\begin{theorem}\label{thm:nerveTheorem}{\bf\rm Nerve Theorem~\cite[\S III.2, p. 59]{Edelsbrunner1999}}.\\
Let $F$ be a finite collection of closed, convex sets in Euclidean space.  Then the nerve of $F$ and the union of the sets in $F$ have the same homotopy type.
\end{theorem}

Nonempty sets with the same homotopy type are homotopy equivalent~\cite[\S III.2, p. 58]{Edelsbrunner1999}.

Let $\Phi(p)$ be a feature vector that describes $p$ point in a topological space $K$.  A descriptive closed Ball with center $p$ and with radius $r$ (denoted $B_{r,\Phi}(p)$) is defined by
\[
B_{r,\Phi}(p) = \left\{q\in K: \norm{\Phi(p)-\Phi(q)} \leq r\right\}.
\]
From this we obtain a descriptive nerve \Cech$_{r,\Phi}(K)$ on $K$ defined by
\[
\mbox{\Cech}_{r,\Phi}(K) = \Nrv\left\{B_{r,\Phi}(p): p\in K\right\} = \left\{B_{r,\Phi}(p): \bigcap B_{r,\Phi}(p)\neq \emptyset\right\}.
\]
A descriptive \Cech\ complex is a collection of descriptive nerves on $K$ (denoted by $\cx_{\Phi} K$).  $K$ is \emph{covered} by $\cx_{\Phi} K$, provided $K\subseteq \mathop{\bigcup}\mbox{\Cech}_{r,\Phi}(K)$. A main result in this paper is the following consequence of Theorem~\ref{thm:nerveTheorem}.

\begin{theorem}\label{thm:nerveSpokeTheorem}
Let $K$ be finite, bounded region of the Euclidean plane covered by a descriptive \Cech\ complex $\cx_{\Phi} K$.
Then the descriptive nerve $\mbox{\Cech}_{r,\Phi}(K)$ of $K$ and the union of the descriptive nerves in $\cx_{\Phi} K$ are homotopy equivalent.
\end{theorem}

\section{Preliminaries}
This section briefly introduces two basic types of proximities, namely, traditional \emph{spatial proximity} and the more recent \emph{descriptive proximity} in the study of computational proximity~\cite{Peters2016CP}.  

\subsection{Strongly Near Cech Complexes}
A pair of nonempty sets in a proximity space are spatially \emph{near} (\emph{close to each other}), provided the sets have one or more points in common or each set contains one or more points that are sufficiently close to each other.  Let $X$ be a nonempty set, $A,B,C\subset X$. E. \u{C}ech~\cite{Cech1966} introduced axioms for the simplest form of proximity $\delta$, which satisfies
%
\begin{description}
\item[{\rm\bf (P1)}] $\emptyset \not\delta A, \forall A \subset X $.
\item[{\rm\bf (P2)}] $A\ \delta\ B \Leftrightarrow B \delta A$.
\item[{\rm\bf (P3)}] $A\ \cap\ B \neq \emptyset \Rightarrow A \near B$.
\item[{\rm\bf (P4)}] $A\ \delta\ (B \cup C) \Leftrightarrow A\ \delta\ B $ or $A\ \delta\ C$. \qquad \textcolor{blue}{$\blacksquare$}
\end{description}

The \Cech\ proximity becomes a Lodato proximity~\cite{Lodato1962}, provided $\delta$ satisfies the \Cech\ proximity axioms and
%
\begin{description}
\item[{\rm\bf (P5)}]  $A\ \delta\ B$ and $\{b\}\ \delta\ C$ for each $b \in B \ \Rightarrow A\ \delta\ C$.
\qquad \textcolor{blue}{$\blacksquare$}
\end{description}

\noindent The pair $(X,\near)$ is called a Lodato proximity space.  We can associate a topology with the space $(X, \delta)$ by considering as closed sets those sets that coincide with their own closure.  For simplicity, a singleton set $\left\{x\right\}$ is denoted by $x\in X$.  Given $A\subset X$, the interior of $A$ is defined by $\Int A = \left\{x\in X: x\ \near\ A\right\}$.

Nonempty sets $A,B$ in a topological space $X$ equipped with the relation $\sn$, are \emph{strongly near} [\emph{strongly contacted}] (denoted $A\ \sn\ B$), provided the sets have at least one point in common.   The strong contact relation $\sn$ was introduced in~\cite{Peters2015JangjeonMSstrongProximity} and axiomatized in~\cite{PetersGuadagni2015stronglyNear},~\cite[\S 6 Appendix]{Guadagni2015thesis}.\\


Let $X$ be a topological space, $A, B, C \subset X$ and $x \in X$.  The relation $\sn$ on the family of subsets $2^X$ is a \emph{strong proximity}, provided it satisfies the following axioms.

\begin{description}
\item[{\rm\bf (snN0)}] $\emptyset\ \sfar\ A, \forall A \subset X $, and \ $X\ \sn\ A, \forall A \subset X$.
\item[{\rm\bf (snN1)}] $A \sn B \Leftrightarrow B \sn A$.
\item[{\rm\bf (snN2)}] $A\ \sn\ B$ implies $A\ \cap\ B\neq \emptyset$. 
\item[{\rm\bf (snN3)}] If $\{B_i\}_{i \in I}$ is an arbitrary family of subsets of $X$ and  $A\ \sn\ B_{i^*}$ for some $i^* \in I \ $ such that $\Int(B_{i^*})\neq \emptyset$, then $  \ A\ \sn\ (\bigcup_{i \in I} B_i)$ 
\item[{\rm\bf (snN4)}]  $\mbox{int}A\ \cap\ \mbox{int} B \neq \emptyset \Rightarrow A\ \sn\ B$.  
\qquad \textcolor{blue}{$\blacksquare$}
\end{description}

\noindent When we write $A\ \sn\ B$, we read $A$ is \emph{strongly near} $B$ ($A$ \emph{strongly contacts} $B$).  The notation $A\ \notfar\ B$ reads $A$ is not strongly near $B$ ($A$ does not \emph{strongly contact} $B$).  A point $p\in A\subset X$ is a boundary point of $A$ (denoted by $\bdy A$), provided any closed ball $B_r(p)$ intersects both $A$ and its complement. For each \emph{strong proximity} (\emph{strong contact}), we assume the following relations:
\begin{description}
\item[{\rm\bf (snN5)}] $x \in \Int A \Rightarrow x\ \sn\ A$.
\item[{\rm\bf (snN6)}] $x \in \bdy A\ \mbox{and}\ A\cap B\neq\emptyset \Rightarrow x\ \sn\ A\ \mbox{and}\ A\ \sn\ B$.  
\item[{\rm\bf (snN7)}] $\{x\}\ \sn \{y\}\ \Leftrightarrow x=y$.  \qquad \textcolor{blue}{$\blacksquare$} 
\end{description}

The pair $(X,\sn)$ is called a \emph{strong proximity space}.  For strong proximity of the nonempty intersection of interiors, we have that $A \sn B \Leftrightarrow \Int A \cap \Int B \neq \emptyset$ or either $A$ or $B$ is equal to $X$, provided $A$ and $B$ are not singletons; if $A = \{x\}$, then $x \in \Int(B)$, and if $B$ too is a singleton, then $x=y$. It turns out that if $A \subset X$ is an open set, then each point that belongs to $A$ is strongly near $A$.  The bottom line is that strongly near sets always share points, which is another way of saying that sets with strong contact have nonempty intersection.   Let $\near$ denote a traditional proximity relation~\cite{Naimpally1970}.   By definition, a \Cech\ nerve \Cech$_r(A)$ is strongly near another \Cech\ nerve \Cech$_r(B)$, provided $B_r(p)\ \sn\ B_r(q)$ for some closed ball $B_r(p)\in$ \Cech$_r(A)$ and some closed ball $B_r(q)\in$ \Cech$_r(B)$.

\begin{proposition}\label{prop:2spoke}
Let \Cech$_r(A)$, \Cech$_r(B)$ be \Cech\ nerves in a strong proximity space $(X,\sn)$.  \Cech$_r(A)\ \sn\ $\Cech$_r(B)$, if and only if $B_r(p)\cap B_r(q) \neq \emptyset$ for some $B_r(p)\in \mbox{\Cech}_r(A), B_r(q)\in \mbox{\Cech}_r(B)$.
\end{proposition}
\begin{proof}$\mbox{}\\
B_r(p)\cap B_r(q) \neq \emptyset$ for some $B_r(p)\in \mbox{\Cech}_r(A), B_r(q)\in \mbox{\Cech}_r(B)$.  If $x\in \bdy B_r(p)$, then, from Axiom (snN6), $B_r(p)\ \sn\ B_r(q)\Rightarrow$(from Axiom (snN2))$\ B_r(p)\cap B_r(q) \neq \emptyset\Leftrightarrow$\ \Cech$_r(A)\ \sn\ $\Cech$_r(B)$.  

Otherwise, $\Int B_r(p)\ \cap\ \Int B_r(q) \neq \emptyset$ implies $B_r(p)\ \sn\ B_r(q)\Rightarrow$(from Axiom (snN2))$\ B_r(p)\cap B_r(q) \neq \emptyset\Leftrightarrow$(from Axiom (snN4))\ \Cech$_r(A)\ \sn\ $\Cech$_r(B)$.
\end{proof}

\begin{corollary}\label{cors:overlappingCechComplexes}
Let \Cech$_r(A)$, \Cech$_r(B)$ be \Cech\ nerves in a strong proximity space $(X,\sn)$.  If a closed ball $B_r(p)\in  
\mbox{\Cech}_r(A)\ \cap\ \mbox{\Cech}_r(A)\neq \emptyset$, then $\mbox{\Cech}_r(A)\ \sn\ \mbox{\Cech}_r(B)$.
\end{corollary}
\begin{proof}
Immediate from Prop.~\ref{prop:2spoke}.
\end{proof}

\begin{example}
Several \Cech\ nerves in a strong proximity space $(X,\sn)$ are represented in Fig.~\ref{fig:CechComplexes}.
Observe that the interior of closed ball $B_r(p_A)\in \mbox{\Cech}_r(A)$ overlaps the interior of the closed ball $B_r(p_B)\in \mbox{\Cech}_r(B)$.  In other words, \Cech$_r(A)$, \Cech$_r(B)$ overlap.  Hence, from Cor.~\ref{cors:overlappingCechComplexes}, we have \Cech$_r(A)\ \sn\ $\Cech$_r(B)$. 
\qquad \textcolor{blue}{$\blacksquare$}  
\end{example}

\begin{figure}[!ht]
\centering
\begin{pspicture}
(-3.0,0.35)(7.0,7.2)
%
%
\pscircle[fillstyle=solid,fillcolor=green!20,opacity=0.5](1.15,4.85){1.2}
\pscircle[fillstyle=solid,fillcolor=gray!10,opacity=0.5](3.05,5.25){1.2}
\pscircle[fillstyle=solid,fillcolor=gray!20,opacity=0.5](2.55,5.85){1.2}
\pscircle[fillstyle=solid,fillcolor=gray!55,opacity=0.5](2.55,4.55){1.2}
\rput(2.2,3.1){$\boldsymbol{\mbox{\Cech}_r(B)}$}
\rput(0.7,5.25){$\boldsymbol{B_r(p_B)}$}
%
%
\pscircle[fillstyle=solid,fillcolor=green!20,opacity=0.5](-0.5,3.45){1.2}
\pscircle[fillstyle=solid,fillcolor=gray!10,opacity=0.5](-1.5,1.9){1.2}
\pscircle[fillstyle=solid,fillcolor=gray!20,opacity=0.5](-2.25,2.9){1.2}
\pscircle[fillstyle=solid,fillcolor=gray!55,opacity=0.5](-0.5,1.9){1.2}
\rput(-1.5,0.5){$\boldsymbol{\mbox{\Cech}_r(A)}$}
\rput(-0.5,3.55){$\boldsymbol{B_r(p_A)}$}
%
%
\pscircle[fillstyle=solid,fillcolor=green!20,opacity=0.5](5.8,2.0){1.2}
\pscircle[fillstyle=solid,fillcolor=gray!20,opacity=0.5](4.5,2.5){1.2}
\pscircle[fillstyle=solid,fillcolor=gray!55,opacity=0.5](5.5,3.5){1.2}
\rput(4.0,1.1){$\boldsymbol{\mbox{\Cech}_r(E)}$}
\rput(6.0,1.5){$\boldsymbol{B_r(p_E)}$}
\end{pspicture}
\caption[]{ Three \Cech\ complexes: \Cech$_r(A)$,\Cech$_r(B)$,\Cech$_r(E)$}
\label{fig:CechComplexes}
\end{figure}

\subsection{Descriptively Proximities on Collections of Cech Nerves}
Descriptive proximities result from the introduction of the descriptive intersection of pairs of nonempty sets. 

\begin{description}
\item[{\rm\bf ($\boldsymbol{\Phi}$)}] $\Phi(A) = \left\{\Phi(x)\in\mathbb{R}^n: x\in A\right\}$, set of feature vectors.
\item[{\rm\bf ($\boldsymbol{\dcap}$)}]  $A\ \dcap\ B = \left\{x\in A\cup B: \Phi(x)\in \Phi(A) \& \in \Phi(x)\in \Phi(B)\right\}$.
\qquad \textcolor{blue}{$\blacksquare$}
\end{description}

\begin{remark}
Depending on the context, the real-valued feature vector restriction $\Phi(A) = \left\{\Phi(x)\in\mathbb{R}^n: x\in A\right\}$ can be relaxed, allowing for other forms of feature vectors (see, for example,~\cite[\S 5]{DiConcilioEtAl2017MCSdescriptiveProximities}). \qquad \textcolor{blue}{$\blacksquare$}
\end{remark}

The descriptive Lodato proximity~\cite[\S 4.15.2]{Peters2013springer} that satisfies the following axioms.

\begin{description}
\item[{\rm\bf (dP0)}] $\emptyset\ \dfar\ A, \forall A \subset X $.
\item[{\rm\bf (dP1)}] $A\ \dnear\ B \Leftrightarrow B\ \dnear\ A$.
\item[{\rm\bf (dP2)}] $A\ \dcap\ B \neq \emptyset \Rightarrow\ A\ \dnear\ B$.
\item[{\rm\bf (dP3)}] $A\ \dnear\ (B \cup C) \Leftrightarrow A\ \dnear\ B $ or $A\ \dnear\ C$.
\item[{\rm\bf (dP4)}] $A\ \dnear\ B$ and $\{b\}\ \dnear\ C$ for each $b \in B \ \Rightarrow A\ \dnear\ C$. \qquad \textcolor{blue}{$\blacksquare$}
\end{description}

\begin{lemma}\label{prop:dnear}
Let $\left(X,\dnear\right)$ be a descriptive proximity space, $A,B\subset X$.  Then $A\ \dnear\ B \Rightarrow A\ \dcap\ B\neq \emptyset$.
\end{lemma}
\begin{proof}
$A\ \dnear\ B \Leftrightarrow$ there is at least one $x\in A, y\in B$ such that $\Phi(x)=\Phi(y)$ (by definition of $A\ \dnear\ B$).  Hence, $A\ \dcap\ B\neq \emptyset$.
\end{proof}

Let $2^X$ denote a collection of subsets of $X$, $2^{2^X}$, a collection of subcollections of $2^X$.

\begin{proposition}
Let $\left(X,\dnear\right)$ be a descriptive proximity space, $\mbox{\Cech}_r(A),\mbox{\Cech}_r(B)\in 2^{2^X}$.  Then $\mbox{\Cech}_r(A)\ \dnear\ \mbox{\Cech}_r(B) \Rightarrow \mbox{\Cech}_r(A)\ \dcap\ \mbox{\Cech}_r(B)\neq \emptyset$.
\end{proposition}
\begin{proof}
$\mbox{\Cech}_r(A)\ \dnear\ \mbox{\Cech}_r(B) \Rightarrow B_r(p_A)\ \dnear\ B_r(p_B)$ for some closed ball $B_r(p_A)\in \mbox{\Cech}_r(A)$ and some closed ball $B_r(p_A)\in \mbox{\Cech}_r(A)$.  Hence, from Lemma~\ref{prop:dnear}, the result follows.
\end{proof}

\noindent Also, from Lemma~\ref{prop:dnear}, we have the following result.


\begin{proposition}
Let $\left(X,\dnear\right)$ be a descriptive proximity space, $\mbox{\Cech}_r(A),\mbox{\Cech}_r(B)\in 2^{2^X}$.  Then 
\begin{compactenum}[1$^o$]
\item $\mbox{\Cech}_r(A)\ \dnear\ \mbox{\Cech}_r(B) \Leftrightarrow \Nrv\left\{B_r(p): p\in A\right\}\ \dnear\ \Nrv\left\{B_r(q): q\in B\right\}$.
\item $\mathop{\bigcap}\limits_{x\in A\cup B} B_r(x)\neq \emptyset$.
\end{compactenum}
\end{proposition}
\begin{proof}$\mbox{}$\\
1$^o$: By definition, $\mbox{\Cech}_r(A)= \Nrv\left\{B_r(x): x\in A\right\}$.  
Similarly, $\mbox{\Cech}_r(B)= \Nrv\left\{B_r(x): x\in B\right\}$. 
Hence, $\Nrv\left\{B_r(p): p\in A\right\}\ \dnear\ \Nrv\left\{B_r(q): q\in B\right\}$.\\
2$^o$: From 1$^o$ and Lemma~\ref{prop:dnear}, the result follows.
\end{proof}

\begin{example}
Let $\mbox{\Cech}_r(A),\mbox{\Cech}_r(B),\mbox{\Cech}_r(E)$ be members of the descriptive Lodato proximity space $\left(2^{2^X},\dnear\right)$.  Let $\Phi(B_r(x))$ equal the surface colour of a closed ball $B_r(x)$.  By definition, a collection of subsets $\mathscr{A}$ is descriptively near another collection of subsets $\mathscr{B}$, provided $A\ \dnear\ B$ for some $A\in \mathscr{A}$ and some $B\in \mathscr{B}$. 
As a result, $\mbox{\Cech}_r(A)\ \dnear\ \mbox{\Cech}_r(E)$, since $\Phi(B_r(p_A)) = \Phi(B_r(p_E))$ in Fig.~\ref{fig:CechComplexes}.

Also observe that $\mbox{\Cech}_r(A)\ \dnear\ \mbox{\Cech}_r(B)$, since $\Phi(B_r(p_A)) = \Phi(B_r(p_B))$ in Fig.~\ref{fig:CechComplexes}.  Hence, from Lemma~\ref{prop:dnear}, $\mbox{\Cech}_r(A)\ \dcap\ \mbox{\Cech}_r(B)$.
\qquad \textcolor{blue}{$\blacksquare$}
\end{example}

Next, consider a proximal form of a Sz\'{a}z relator~\cite{Szaz1987}.  A \emph{proximal relator} $\mathscr{R}$ is a set of proximity relations on a nonempty set $X$~\cite{Peters2016relator}.  The pair $\left(X,\mathscr{R}\right)$ is a proximal relator space.  The connection between $\sn$ and $\near$ is summarized in Prop.~\ref{thm:sn-implies-near}.

\begin{lemma}\label{thm:sn-implies-near}
Let $\left(X,\left\{\near,\dnear,\sn\right\}\right)$ be a proximal relator space, $A,B\subset X$.  Then 
\begin{compactenum}[1$^o$]
\item $A\ \sn\ B \Rightarrow A\ \near\ B$.
\item $A\ \sn\ B \Rightarrow A\ \dnear\ B$.
\end{compactenum}
\end{lemma}
\begin{proof}$\mbox{}$\\
1$^o$: From Axiom (snN2), $A\ \sn\ B$ implies $A\ \cap\ B\neq \emptyset$, which implies $A\ \near\ B$ (from Lodato Axiom (P2)).\\
2$^o$: From 1$^o$, there are $x\in A, y\in B$\ common to $A$ and $B$.  Hence, $\Phi(x) = \Phi(y)$, which implies $A\ \dcap\ B\neq \emptyset$.  Then, from the descriptive Lodato Axiom (dP2), $A\ \dcap\ B \neq \emptyset \Rightarrow\ A\ \dnear\ B$. This gives the desired result.
\end{proof}

\begin{theorem}\label{thm:spoke}
Let $\left(X,\left\{\near,\dnear,\sn\right\}\right)$ be a proximal relator space, $\mbox{\Cech}_r(A),\mbox{\Cech}_r(B)\in 2^{2^X}$.  Then
\begin{compactenum}[1$^o$]
\item $\mbox{\Cech}_r(A)\ \sn\ \mbox{\Cech}_r(B)$ implies $\mbox{\Cech}_r(A)\ \dnear\ \mbox{\Cech}_r(B)$.
\item A closed ball $B_r(x)\ \in\ \mbox{\Cech}_r(A)\cap \mbox{\Cech}_r(B)$ implies if $B_r(x) \in \mbox{\Cech}_r(A)\ \dcap\ \mbox{\Cech}_r(B)$.
\item A closed ball $B_r(x)\ \in\ \mbox{\Cech}_r(A)\cap \mbox{\Cech}_r(B)$ implies $\mbox{\Cech}_r(A)\ \dnear\ \mbox{\Cech}_r(B)$.
\end{compactenum}
\end{theorem}
\begin{proof}$\mbox{}$\\
1$^o$: Immediate from part 2$^o$ Lemma~\ref{thm:sn-implies-near}.\\
2$^o$: From Prop.~\ref{prop:2spoke}, $B_r(x)\ \in\ \mbox{\Cech}_r(A)\ \dcap\ \mbox{\Cech}_r(B)$, if and only if $\mbox{\Cech}_r(A)\ \sn\ \mbox{\Cech}_r(B)$.  Consequently, there are members of $B_r(x)$ common to $\mbox{\Cech}_r(A),\mbox{\Cech}_r(B)$, which have the same description. Hence, $\Phi(B_r(x))\ \in\ \mbox{\Cech}_r(A)\ \dcap\ \mbox{\Cech}_r(B)$.\\
3$^o$: Immediate from 2$^o$ and Lemma~\ref{thm:sn-implies-near}.
\end{proof}


\subsection{Strong Descriptive Proximities on Collections of Cech Nerves}

The descriptive strong proximity $\snd$ is the descriptive counterpart of $\sn$. 

\begin{definition}\label{def:snd}
Let $X$ be a topological space, $A, B, C \subset X$ and $x \in X$.  The relation $\snd$ on the family of subsets $2^X$ is a \emph{descriptive strong Lodato proximity}, provided it satisfies the following axioms.
%
\begin{description}
\item[{\rm\bf (dsnN0)}] $\emptyset\ {\sdfar}\ A, \forall A \subset X $, and \ $X\ \snd\ A, \forall A \subset X$
\item[{\rm\bf (dsnN1)}] $A\ \snd\ B \Leftrightarrow B\ \snd\ A$
\item[{\rm\bf (dsnN2)}] $A\ \snd\ B \Rightarrow\ A\ \dcap\ B \neq \emptyset$
\item[{\rm\bf (dsnN3)}] If $\{B_i\}_{i \in I}$ is an arbitrary family of subsets of $X$ and  $A\ \snd\ B_{i^*}$ for some $i^* \in I \ $ such that $\Int(B_{i^*})\neq \emptyset$, then $A\ \snd\ (\bigcup_{i \in I} B_i)$
\item[{\rm\bf (dsnN4)}] $\Int A\ \dcap\ \Int B \neq \emptyset \Rightarrow A\ \snd\ B$
\qquad \textcolor{blue}{$\blacksquare$}
\end{description}
\end{definition}

\noindent When we write $A\ \snd\ B$, we read $A$ is \emph{descriptively strongly near} $B$.  The notation $A\ \sdfar\ B$ reads $A$ is not descriptively strongly near $B$.
For each \emph{descriptive strong proximity}, we assume the following relations:
\begin{description}
\item[(dsnN5)] $\Phi(x) \in \Phi(\Int (A)) \Rightarrow x\ \snd\ A$.
\item[(dsnN6)] If $\Phi(x) \in \Phi(\Int (A)\ \mbox{\&}\ \Phi(x) \in \Phi(\Int (B) \Rightarrow A\ \snd\ B$.
\item[(dsnN7)] $\{x\}\ \snd\ \{y\} \Leftrightarrow \Phi(x)=\Phi(y)$.  \qquad \textcolor{blue}{$\blacksquare$} 
\end{description}

So, for example, if we take the strong proximity related to non-empty intersection of interiors, we have that $A\ \snd\ B \Leftrightarrow \Int A\ \dcap\ \Int B \neq \emptyset$ or either $A$ or $B$ is equal to $X$, provided $A$ and $B$ are not singletons; if $A = \{x\}$, then $\Phi(x) \in \Phi(\Int(B))$, and if $B$ is also a singleton, then $\Phi(x)=\Phi(y)$.

\begin{lemma}\label{cor:sndCechComplexes}
Let $\Int B_r(p)\in \mbox{\Cech}_r(A),\Int B_r(q)\in \mbox{\Cech}_r(B)$ be closed balls in \Cech\ nerves in a relator space $\left(X,\left\{\dnear,\snd\right\}\right)$.  Then
\begin{compactenum}[1$^o$]
\item $\Int B_r(p)\ \dnear\ \Int B_r(q)$ implies $\mbox{\Cech}_r(A)\ \dnear\ \mbox{\Cech}_r(B)$.
\item $\mbox{\Cech}_r(A)\ \dnear\ \mbox{\Cech}_r(B)$ implies $\mbox{\Cech}_r(A)\ \snd\ \mbox{\Cech}_r(B)$.
\end{compactenum}
\end{lemma}
\begin{proof}$\mbox{}$\\
1$^o$: The result follows from Axiom (dsnN6), since $\Int B_r(p), \Int B_r(p)$ have at least one member in common with the same description.\\
2$^o$: From Axiom (dsnN4), $\mbox{\Cech}_r(A)\ \snd\ \mbox{\Cech}_r(B)$.
\end{proof}

\begin{theorem}\label{lem:descriptivelyNearNerves}
\Cech\ nerves containing strongly near closed ball interiors are strongly descriptively near in a relator space $\left(X,\left\{\sn,\dnear,\snd\right\}\right)$.
\end{theorem}
\begin{proof}
Let $\Int B_r(p)\ \snd\ \Int B_r(q)$ for $\Int B_r(p)\in \mbox{\Cech}_r(A), \Int B_r(q)\in \mbox{\Cech}_r(B)$.  From Axiom (dsnN2), $\Int B_r(p)\ \cap\ \Int B_r(q)$.  Consequently, $\mbox{\Cech}_r(A)\ \sn\ \mbox{\Cech}_r(B)$.  Hence, from Lemma~\ref{thm:sn-implies-near}, $\mbox{\Cech}_r(A)\ \dnear\ \mbox{\Cech}_r(B)$.  Hence, from Lemma~\ref{cor:sndCechComplexes}, the result follows.
\end{proof}

\noindent The interior of a \Cech\ nerve (denoted by $\Int \mbox{\Cech}_r(X)$) is defined by
\[
\Int \mbox{\Cech}_r(X) = \left\{x\in X: x\ \sn\ \mathop{\bigcup}\limits_{B_r(p)\in \mbox{\Cech}_r(X)}\Int B_r(p)\right\}.
\]

\begin{theorem}\label{lem:stronglyDescriptivelyNearNerves}
Let $\mbox{\Cech}_r(A),\mbox{\Cech}_r(B)$ be \Cech\ nerves in a proximal relator space $\left(X,\left\{\sn,\dnear,\snd\right\}\right)$.
If $\Int \mbox{\Cech}_r(A)\ \dcap\ \Int \mbox{\Cech}_r(B)$, then $\mbox{\Cech}_r(A)\ \snd\ \mbox{\Cech}_r(B)$.
\end{theorem}
\begin{proof}
Immediate from Theorem~\ref{lem:descriptivelyNearNerves}.
\end{proof}

\section{Main Results}
Homotopy types~\cite[\S III.2]{Edelsbrunner1999},~\cite[\S 5.3]{Peters2016excursions} lead to significant results in the theory of shapes and nerve in complexes covering a shape.  The flexibility provided by descriptive \Cech\ nerves makes it possible to cover the interior as well as the contour of shapes, since the neighbourhoods of sets of simplexes of such nerve complexes can easily fill irregular shape interiors that change and deform over time.  The shapes of descriptive neighbourhoods in \Cech\ nerves tend to be non-uniform with features such as diameter, perimeter length, convexity and boundedness. Such shape-filling nerves have a persistence utility that stays in existence, provided the amount of contour overlap is below some threshold and the strong descriptive proximity of complexes covering a particular shape, is very high.  

This proposed approach to shape theory ushers in a very practical form of persistence homology~\cite[\S VII]{Edelsbrunner1999},~\cite{EdelsbrunnerMorozov2013} useful in applications such as the study of shape in proteins and protein complexes~\cite{Agarwal2006matchingShapesInProteinDocking,Ban2004interfaceSurfacesProteinProteinComplexes}, plant root structure~\cite{Galkovsky2012plantrootShape}, speech patterns~\cite{BrownKnudson2009humanSpeech}, and digital image compression and segmentation~\cite{Carlsson2008naturalImageSpaces,Edelsbrunner2003MorseSmaleComplexes}, neuroscience~\cite{Dabaghian2012hippocampalSpatialMapPersistence}, orthodontia~\cite{GambleHeo2010landmarkShapePersistence}, gene expression~\cite{Dequeant2008timeSeriesSegmentationClock} and especially in the study of time-varying shapes~\cite{Munch2013PhDpersistentHomology}.  In a persistence homology on time varying systems, the equivalence and types of the shapes of descriptive nerves is important, since we need to know when one shape-filling nerve complex has a persistence that is similar to that of another nerve complex.

Let $f,g:X\longrightarrow Y$ be two continuous maps.  A \emph{homotopy} between $f$ and $g$ is a continuous map $H:X\times[0,1]\longrightarrow Y$ so that $H(x,0) = f(x)$ and $H(x,1) = g(x)$.  The sets $X$ and $Y$ are \emph{homotopy equivalent}, provided there are continuous maps $f: X\longrightarrow Y$ and $g:Y\longrightarrow X$ such that $g\circ f \simeq \mbox{id}_X$ and $f\circ g \simeq \mbox{id}_Y$.  This yields an equivalence relation $X\simeq Y$.  In addition, $X$ and $Y$ have the same \emph{homotopy type}, provided $X$ and $Y$ are homotopy equivalent.  

\begin{lemma}\label{lem:MNCnerves}
Let $\mbox{\Cech}_r(K)$ be a \Cech\ nerve on $K$ endowed with the strong proximity $\sn$.  Then $\mathop{\bigcap}\limits_{p\in K} B_r(p)\neq \emptyset$.
\end{lemma}
\begin{proof}
Closed balls $B_r(p), B_r(q)\in \mbox{\Cech}(r)$ overlap, {\em i.e.}, $B_r(p)\ \sn\ B_r(q)\Rightarrow B_r(p)\cap B_r(q)\neq\emptyset$ (from Axiom (snN2)).  Hence, the result follows.
\end{proof}

\begin{theorem}\label{thm:sn-nerve}
Let $\left(\mbox{\Cech}(r),\left\{\near,\dnear,\sn\right\}\right)$ be a proximal relator space, closed balls $B_r(p), B_r(q)\in \mbox{\Cech}(r)$.  Then 
\begin{compactenum}[1$^o$]
\item $B_r(p)\ \sn\ B_r(q) \Rightarrow B_r(p)\ \near\ B_r(q)$.
\item $B_r(p)\ \sn\ B_r(q) \Rightarrow B_r(p)\ \dnear\ B_r(q)$.
\end{compactenum}
\end{theorem}
\begin{proof}$\mbox{}$\\
1$^o$: From Lemma~\ref{lem:MNCnerves}, $\mbox{\Cech}(r)$ is a \Cech\ nerve.  
$B_r(p)\ \sn\ B_r(q)$ for every pair of closed balls $B_r(p), B_r(q)$ in the nerve $\mbox{\Cech}(r)$.  From Lemma~\ref{lem:MNCnerves}, $\mathop{\bigcap}\limits_{B_r(p)\in \Nrv K} B_r(p)\neq \emptyset$.  Consequently, from Axiom (snN2), $B_r(p)\ \near\ B_r(q)$ for all closed balls $B_r(p), B_r(q)\in \mbox{\Cech}(r)$.\\
2$^o$: Closed balls $B_r(p), B_r(q)\in \mbox{\Cech}(r)$ overlap.  Hence, $B_r(p)\ \dcap\ B_r(q)\neq \emptyset$.  Then, from Lemma~\ref{thm:sn-implies-near}, $B_r(p)\ \dcap\ B_r(q) \neq \emptyset \Rightarrow\ B_r(p)\ \dnear\ B_r(q)$. This gives the desired result for each pair of closed balls in the \Cech\ nerve.
\end{proof}

\begin{lemma}\label{thm:1spokeHomotopy}
Let  $\cx K$ be a collection of closed balls, which are convex sets covering a finite bounded region $K$ in the Euclidean plane.  Then the nerve $\mbox{\Cech}_r(K)$ and the union of the closed balls in $\cx K$ have the same homotopy type.
\end{lemma}
\begin{proof}
From Theorem~\ref{thm:nerveTheorem}, the union of the closed balls $B_r(p\in K)\in \cx K$ and nerve $\mbox{\Cech}_r(K)$ have the same homotopy type.
\end{proof}

\begin{remark}
From Lemma~\ref{thm:1spokeHomotopy}, the nerve $\mbox{\Cech}_r(p), p\in K$ and $\mathop{\bigcup}\limits_{p\in K}\mbox{\Cech}_{r}(p) = \cx K$ are homotopy equivalent.  Consequently, the descriptive nerve $\mbox{\Cech}_{r,\Phi}(p\in K)\in \cx_{\Phi} K$ and $\mathop{\bigcup}\limits_{p\in K}\mbox{\Cech}_{r,\Phi}(p) = \cx_{\Phi} K$ are also homotopy equivalent.  This proves Theorem~\ref{thm:nerveSpokeTheorem}. \qquad \textcolor{blue}{$\blacksquare$}
\end{remark}

\bibliographystyle{amsplain}
\bibliography{NSrefs}

\end{document}